\documentclass[12pt]{article}
\usepackage[ansinew]{inputenc}
\usepackage{array}
\usepackage{color}
\usepackage{amsmath}
\usepackage{amsxtra}
\usepackage{amstext}
\usepackage{amssymb}
\usepackage{latexsym}
\usepackage{bibentry}
\newcommand{\proof}{{\it Proof~}}
\newtheorem{thm}{Theorem}[section]

\newtheorem{lem}[thm]{Lemma}

\newtheorem{defn}[thm]{Definition}

\numberwithin{equation}{section}

\begin{document}

\centerline{{\bf  \Large
CR-Lightlike Submanifolds of a Golden Semi-}} 
\centerline{{\bf  \Large Riemannian Manifold}}

\vspace*{.2cm}

\begin{center}

{Mobin Ahmad} \\
Department of Mathematics and Statistics \\
Faculty of Science,
Integral University, \\
Lucknow-226026, India \\
E-mail: \small{mobinahmad68@gmail.com}\\ 
\end{center}

\begin{center}
 {Mohammad Aamir Qayyoom}\\	
 Department of Mathematics and Statistics \\
 Faculty of Science,
 Integral University, \\
 Lucknow-226026, India \\	
 E-mail: \small{aamir.qayyoom@gmail.com}
\end{center}
\vspace*{.2cm}

\begin{abstract}
In this paper, we define and study CR-lightlike submanifolds of a golden semi-Riemannian manifold. We investigate some properties of geodesic CR-submanifolds of a golden semi-Riemannian manifold. Moreover, we obtain many interesting results for totally geodesic and totally umbilical CR-submanifolds on a golden Riemannian manifold. A non-trivial example of CR-lightlike submanifold of a golden semi-Riemannian manifold is also constructed.
\end{abstract}



\section{Introduction}

\indent 

The theory of submanifolds of a manifold is one of the most interesting topics in differential geometry. According to the behaviour of the tangent bundle of a
submanifold, we have three classes of submanifolds: holomorphic submanifolds, totally real submanifolds and CR-submanifolds has been introduced by the author in \cite{bdecby1978}. 
Let $ \overline{M} $ be an almost Hermitian manifold and let $J$ be the almost complex structure of $\overline{M}.$ A real submanifold $M$ of $\overline{M}$ is called a CR submanifold if there exists a differential distribution $D$ on $M$ satisfying \\
$$(i) J(D_{x}) = D_{x}  $$ and
$$(ii) J(D_{x}^{\bot}) \subset T_{x}M^{\bot}  $$
for each $x \in M,$ where $D^{\bot}$ is the complementary orthogonal distribution to $D$ and $T_{x}M^{\bot}$ is normal space to $M$ at $x.$ Holomorphic submanifold and totally real submanifolds are particular cases of CR-submanifolds.\\
Later, the CR-submanifolds have been extensively studied by several geometers  ( see in \cite{amakk2013}, \cite{amak2013}, \cite{bakmyk1981}, \cite{beja1981}, \cite{beja1986}  \cite{chen1981}). Also, Some properties of CR-submanifold have been investigated in various studies (see \cite{amhajb2014}, \cite{amoj2010}, \cite{ahmd2012}, \cite{amkk2016}, \cite{amak2013}, \cite{aaik2013}, \cite{ammdsr2012}, \cite{dhsi1986}, \cite{dlam2009}, \cite{krgmnrk2010}, \cite{smdmaoj2014}, \cite{zksa2016}).\\
\indent
Bejancu and Duggal \cite{dklba1998} and Kupeli \cite{kudn1996} developed the general theory of lightlike submanifolds of a semi-Riemannian manifold. In \cite{beja1995}, the authors constructed the principal vector bundles to  lightlike submanifolds which has a non-trivial intersection with the tangent bundle of lightlike submanifolds in a  semi-Riemann manifold and obtained Gauss-Weingarten formulae as well as other properties of this submanifold. Duggal and Bejancu \cite{dklba1993} studied lightlike hypersurfaces of indefinite Kaehler manifolds. Duggal and Sahin \cite{dklsb2007} studied lightlike submanifolds of indefinite Sasakian manifolds. Acet et. al \cite{acpk2014} studied lightlike submanifolds of a para-Sasakian manifold. Sahin and Gunes \cite{sbgr2001} studied the geodesic CR-lightlike submanifolds in a Kaehler manifold. They also studied the integrability of distributions in CR-lightlike submanifolds in \cite{sbgr2002}. Gonga et. al \cite{gmkn2010} studied totally umbilical CR-lightlike submanifolds of indefinite Kaehler manifolds.

\indent
On the other hand, Crasmareanu and Hretcanu \cite{cmhc2008} constructed the golden structure on a differentiable manifold $(\overline{M}, g)$ as a particular case of polynomial structure \cite{goldyk1970}. Gezer et. al investigated the integrability conditions of golden Riemannian structure. The golden structure was also studied in various studies (see \cite{efsrua2017}, \cite{gacnsa2013}, \cite{gheb2018}, \cite{goke2019}, \cite{qmama2022}). Hretcanu C.E. \cite{hrce2007} studied submanifolds in Riemannian manifold with golden structure. The authors studied some properties of submanifoldsof Riemannian manifolds with golden structure in \cite{ahqa2019}. They also studied hypersurfaces of a metallic Riemannian manifold in \cite{ahqa2020}.\\
\indent
The growing importance of lightlike submanifolds in mathematical physics, in particular, their use in relativity and many more, motivated the authors to study lightlike submanifolds extensively. In this paper, we study CR-lightlike submanifold of a golden semi-Riemannian manifold. The paper is organized as follows:\\
In section 2, we define a golden structure manifold and CR-lightlike submanifolds of a golden Riemannian manifold. In section 3, we establish several properties of geodesic CR-lightlike submanifolds on a golden Riemannian manifold. In section 4, we investigate some properties of totally umbilical CR-lightlike submanifolds. In the last section, we give a non-trivial example of CR-submanifolds of a golden semi-Riemannian manifold.

\section{Definition and preliminaries}
Let $ ( \overline{M}, {g} ) $ is a $ (m + n)-$dimensional semi-Riemannian manifold, $m, n \ge 1$ and $ {g} $ be a semi-Riemannian metric on $\overline{M}.$ We denote by $q$ the constant index of ${g}$ and we suppose that $\overline{M}$ is not Riemannian manifold. \\
Let $M$ be a lightlike submanifold of dimension $m$ of $\overline{M}.$ In this case there exists a smooth distribution on $M,$ named a radical distribution $ Rad(TM)$ such that $ Rad(TM) = TM\cap TM^{\bot},$ $\forall p \in M.$ If the rank of $RadTM$ is $r \ge 0,$ $M$ is called an $r-$lightlike submanifold of $\overline{M}.$ Then there arise four cases:\\
I: $ 0 < r < min \{ m, n\}; $ the submanifold is called an $r-$lightlike submanifold. \\
II: $ 1 < r = n < m; $  The submanifold is called a coisotropic submanifold.\\
III: $ 1 < r = m < n; $ The submanifold is called a isotropic submanifold.\\
IV: $ 1 < r = n = m; $ The submanifold is called totally lightlike submanifold.\\
Let $M$ be an $r-$lightlike submanifold of $\overline{M}.$ We consider the complementary distribution $S(TM)$ of $Rad(TM)$ on $TM$ which is called a screen distribution. Then we have the direct orthogonal sum
\begin{equation}\label{1e1}
\nonumber TM = RadTM \bot S(TM).
\end{equation}
As, $S(TM)$ is a non-degenerate vector subbundle of $T\overline{M}|_{M},$ we put
\begin{equation}\label{1e2}
\nonumber T\overline{M}|_{M} = S(TM) \bot S(TM)^{\bot},
\end{equation}
 where $S(TM)^{\bot}$ is the complementary orthogonal vector subbundle of $S(TM)$ in $T\overline{M}|_{M}.$ \\
 $S(TM), S(TM^{\bot})$ are non-degenerate, we have the following orthogonal direct decomposition
 \begin{equation}\label{1e3}
\nonumber  S(TM)^{\bot} = S(TM^{\bot}) \bot S(TM^{\bot})^{\bot}.
 \end{equation}

 \begin{thm}\label{the1}\cite{dklba1998}
 Let $(M, g, S(TM), S(TM)^{\bot})$ be an $r-$lightlike submanifold of a semi-Riemannian manifold $(\overline{M}, {g}).$ Then there exists a complementary vector bundle $ltr(TM)$ called a lightlike transversal bundle  of $Rad(TM)$ in $S(TM^{\bot})^{\bot}$ and the basis of $\Gamma (ltr(TM)|_{V})$ consists of smooth sections $\{ N_{1},....., N_{r} \}$ of $S(TM^{\bot})^{\bot}|_{V}$ such that
 $$ {g}(N_{i}, \xi_{j}) = \delta_{ij}, \; {g}(N_{i}, N_{j}) = 0, \; i,j = 0, 1, 2,,,, r, $$
 where $\{ \xi_{1}, ....., \xi_{r} \}$ is a basis of $\Gamma (RadTM)|_{U}.$
 \end{thm}

 We consider the vector bundle
 \begin{equation}\label{1e4}
 \nonumber tr(TM) = ltr(TM) \bot S(TM^{\bot}).
 \end{equation}
 Thus,
 \begin{equation}\label{1e5}
\nonumber  T\overline{M} = TM \oplus tr(TM) = S(TM) \bot S(TM^{\bot}) \bot (Rad(TM)\oplus ltr(TM).
 \end{equation}
Let $\overline{\nabla}$ be the Levi-Civita connection on $\overline{M},$ we have
\begin{equation}\label{1e5a}
\overline{\nabla}_{X}Y = \nabla_{X}Y + h(X,Y),\; \forall X,Y \in \Gamma (TM) 
\end{equation} and
\begin{equation}\label{1e5b}
\overline{\nabla}_{X}V = - A_{V}X + \nabla_{X}^{\bot}V, \; \forall X \in \Gamma (TM)\; and \; V \in \Gamma (tr(TM)),
\end{equation} 
where $\{ \nabla_{X}Y, A_{V}X \}$ and $\{ h(X,Y), \nabla_{X}^{\bot}V \}$ belongs to $\Gamma (TM)$ 
and $\Gamma (tr(TM)),$ respectively. \\ 
Using projection $L : tr(TM) \rightarrow ltr(TM),$ and $S: tr(TM) \rightarrow S(TM^{\bot}),$  we have
\begin{equation}\label{1e6}
\overline{\nabla}_{X}Y = \nabla_{X}Y + h^{l}(X,Y) + h^{s}(X,Y),
\end{equation}
\begin{equation}\label{1e7}
\overline{\nabla}_{X}N = - A_{N}X + \nabla_{X}^{l}N + D^{s}(X,N),
\end{equation}
and
\begin{equation}\label{1e8}
\overline{\nabla}_{X} W = - A_{W}X + \nabla_{X}^{s} + D^{l}(X,W)
\end{equation}
for any $X, Y \in \Gamma (TM), \; N \in \Gamma (ltr(TM))$ and $W \in \Gamma (S(TM^{\bot})),$ \\
where $h^{l}(X,Y) = Lh(X,Y), \; h^{s}(X,Y) = Sh(X,Y), \; \nabla_{X}^{l}N, \; D^{l}(X,W) \in \Gamma (ltr(TM)), \; \nabla_{X}^{s}\; D^{s}(X,N) \in \Gamma (S(TM^{\bot})) $ and $\nabla_{X}Y, \; A_{N}X, \; A_{W}X \in \Gamma (TM).$\\
 $P$ denotes the projection morphism of $TM$ to the screen distribution, we consider the distribution
\begin{equation}\label{1e9}
\nonumber \nabla_{X}PY = \nabla_{X}^{*}PY + h^{*}(X,PY),
\end{equation}
\begin{equation}\label{1e10}
\nabla_{X}\xi = - A_{\xi}^{*}X + \nabla_{X}^{*t}\xi,
\end{equation}
where $X, Y \in \Gamma(TM), \; \xi \in \Gamma(Rad(TM)). $  \\
Then, we have the following equations  \\
\begin{equation}\label{1e10a}
\nonumber {g}(h^{l}(X, PY), \xi) = g(A_{\xi}^{*}X, PY),
\end{equation}
\begin{equation}\label{1e10b}
	{g}(h^{*}(X,PY), N) = g(A_{N}X, PY),
\end{equation}
\begin{equation}\label{1e10c}
\nonumber g(A_{\xi}^{*}PX, PY) = g(PX, A_{\xi^{*}}PY ),
\end{equation}
\begin{equation}\label{1e10d}
\nonumber  A^{*}_{\xi}\xi = 0.
\end{equation}
Consider that $\overline{\nabla}$ is a metric connection, we have 
\begin{equation}\label{1e10e}
(\nabla_{X}{g})(Y,Z) = {g}(h^{l}(X,Y), Z)+ {g}(h^{l}(X,YZ), Y).
\end{equation}
Using the properties of linear connection, we have
\begin{equation}\label{1e10f}
(\nabla_{X}h^{l})(Y,Z)=\nabla_{X}^{l}(h^{l}(Y,Z))- h^{l}(\overline{\nabla}_{X}Y,Z)-h^{l}(Y,\overline{\nabla}_{X}Z),
\end{equation}
\begin{equation}\label{1e10g}
	(\nabla_{X}h^{s})(Y,Z)=\nabla_{X}^{s}(h^{s}(Y,Z))- h^{s}(\overline{\nabla}_{X}Y,Z)-h^{s}(Y,\overline{\nabla}_{X}Z).
\end{equation}
\begin{defn}\label{de1}
Let $(\overline{M}, {J}, {g})$ be a real m-dimensional golden semi-Riemannian manifold and $M$ be a real n-dimensional submanifold of $\overline{M},$ that is
\begin{equation}\label{1e16a}
{J}^{2} = {J} + I, 
\end{equation}  
where ${J}$ is a non-null tensor of type $(1,1)$ and metric ${g}$ is ${J}-$compatible, $i.e$ 
\begin{equation}\label{1e16b}
{g} ({J}X, Y) = {g}(X, {J}X)
\end{equation}
then, $\overline{M}$ is called a locally golden manifold if $J$ is parallel with respect to $\overline{\nabla}$,  i.e. $\overline{\nabla}_{X}J = 0$ $\forall X,Y \in \Gamma (T\overline{M}).$
\end{defn}
\begin{defn}\label{de2}
A submanifold $M$ of a golden semi-Riemannian manifold $\overline{M}$ is said to be a CR-lightlike submanifold if the following two conditions are fulfilled: \\
(i) ${J}(Rad(TM))$ is a distribution on $M$ such that
$$ Rad(TM) \cap {J} Rad(TM) = \{ 0 \}. $$
(ii) There exist vector bundles $S(TM), \; S(TM^{\bot}), \; ltr(TM), \; D_{0} \; and \; D' $ over $M$ such that
$$ S(TM) = \{ {J}(RadTM) \oplus D'\}\bot D_{0}, \; {J}D_{0} = D_{0}, \; {J}D' = L_{1}\bot L_{2}, $$
where $D_{0}$ is a non-degenerate distribution on $M$ and $L_{1},$ $L_{2},$ are vector bundles of $ltr(TM)$ and $S(TM^{\bot})$ respectively.
\end{defn}

From the definition of CR-lightlike submanifold, we have
$$ TM = D \oplus D', $$
where $D = Rad(TM) \bot {J} Rad(TM) \bot D_{0}.$ \\
We denote by $S$ and $Q$ the projection on $D$ and $D',$ respectively. Then we have
\begin{equation}\label{1e11}
\nonumber {J}X  = fX + wX
\end{equation}
for any $X, Y \in \Gamma (TM),$ where $fX = {J}SX$ and $wX = {J}QX.$\\
On the other hand, we have
\begin{equation}\label{1e12}
\nonumber {J}V = BV + CV
\end{equation}
for any $V \in \Gamma (tr(TM)),$ where $BV \in \Gamma (TM) $ and $CV \in \Gamma (tr(TM)).$ Unless $M_{1}$ and $M_{2}$ are supposed to as ${J}L_{1}$ and ${J}L_{2},$ respectively.
\begin{lem}\label{2l4}
Let $M$ be a CR-lightlike submanifold of golden semi-Riemannian manifold and screen distribution be totally geodesic. Then $\nabla_{X}Y \in \Gamma (S(TM)),$ where $X, Y \in \Gamma (S(TM)).$ 
\end{lem}
\begin{proof}
For any $X, Y \in \Gamma (S(TM)),$ \\
\begin{eqnarray}
\nonumber {g}(\nabla_{X}Y, N) & = & {g} (\overline{\nabla}_{X}Y - h(X,Y), N)\\
\nonumber & = & {g} (\overline{\nabla}_{X}Y, N) - {g}(h(X,Y), N)\\
\nonumber & = & {g} (\overline{\nabla}_{X}Y, N)\\
\nonumber & = & - {g} (Y,\overline{\nabla}_{X} N).
\end{eqnarray}
Using (\ref{1e7}), we have
\begin{eqnarray}
\nonumber {g}(\nabla_{X}Y, N) & = & - {g}(Y, -A_{N}X + \nabla_{X}^{\bot}N) \\
\nonumber & = & {g}(Y, A_{N}X) - {g}(Y, \nabla_{X}^{\bot}N)\\
\nonumber & = & {g}(Y, A_{N}X).
\end{eqnarray}
Using (\ref{1e10b}), we get 
$$ {g}(\nabla_{X}Y, N) = {g}(h^{*}(X,Y), N).  $$ 
Since screen distribution is totally geodesic, then $h^{*}(X,Y) = 0,$
$$ {g} (\overline{\nabla}_{X}Y, N) = 0. $$
Using Theorem \ref{the1}, we have \\
$$ \nabla_{X}Y \in \Gamma (S(TM)), $$
where $X, Y \in \Gamma (S(TM)). $
\end{proof}

\begin{lem}
Let $M$ be a CR-lightlike submanifold a locally golden semi-Riemannian manifold $\overline{M}.$ 
Then $\nabla_{X}JX = J\nabla_{X}X$ for any $X \in \Gamma (D_{0}).$
\end{lem}

\begin{proof} 
Let $X,Y \in \Gamma (D_{0})$, using (\ref{1e5a}), we have 
\begin{eqnarray}
\nonumber {g}(\nabla_{X}JX, Y) & = & {g}(\overline{\nabla}_{X}JX-h(X, JX),Y)\\
\nonumber & = & {g} (\overline{\nabla}_{X}JX, Y) - {g} (h(X, JX), Y)\\
\nonumber & = & {g} (\overline{\nabla}_{X}JX, Y) \\
\nonumber & = & {g} ((\overline{\nabla}_{X}J)X + J(\overline{\nabla}_{X}X), Y).
\end{eqnarray}
Since $\overline{\nabla}_{X}J=0,$ 
\begin{eqnarray}
\nonumber {g}(\nabla_{X}JX, Y) & = & {g}(J(\overline{\nabla}_{X}X), Y) \\
\nonumber & = & {g}(J(\overline{\nabla}_{X}X), Y)\\
\nonumber & = & {g}((\overline{\nabla}_{X}X), JY)\\
\nonumber & = & {g}(\nabla_{X}X, JY) - {g}(h(X,X), JY) \\
\nonumber & = & {g}(\nabla_{X}X, JY)\\
\nonumber {g}(\nabla_{X}JX, Y) & = & {g}(J(\nabla_{X}X), Y), \\
\nonumber {g}(\nabla_{X}JX - J(\nabla_{X}X), Y) & = & 0.
\end{eqnarray}
 Non-degeneracy of $D_{0},$ implies \\
$$ \nabla_{X}JX = J(\nabla_{X}X),$$ 
where $X \in \Gamma (D_{0}).$
\end{proof}

\section{ Geodesic CR-lightlike submanifolds }

\begin{defn}\label{3d1}
A CR-lightlike submanifold of a golden semi-Riemannian manifold is called mixed geodesic CR-lightlike submanifold if its second fundamental form $h$ satisfies 
$$h(X,U)=0, $$ 
where $X \in \Gamma (D)$ and $U \in \Gamma (D').$
\end{defn}

\begin{defn}\label{3d2}
	A CR-lightlike submanifold of a golden semi-Riemannian manifold is called $D-$geodesic CR-lightlike submanifold if its second fundamental form $h$ satisfies 
	$$h(X,Y)=0, $$ 
	where $X, Y \in \Gamma (D).$
\end{defn}

\begin{defn}\label{3d3}
	A CR-lightlike submanifold of a golden semi-Riemannian manifold is called $D'-$geodesic CR-lightlike submanifold if its second fundamental form $h$ satisfies 
	$$h(U,V)=0, $$ 
	where $U, V \in \Gamma (D').$
\end{defn}

\begin{thm}\label{2th1}
Let $\overline{M}$ be a golden semi-Riemannian manifold and $M$ be a CR-lightlike submanifold of $\overline{M}.$ Then, M is totally geodesic if and only if
$$ (L_{\xi}{g})(X,Y) = 0 $$ and
$$ (L_{W}{g})(X,Y) = 0 $$
for any $X, Y \in \Gamma (TM), \; \xi \in \Gamma (Rad(TM)) \; and \; W \in \Gamma (S(TM^{\bot})).$
\end{thm}
\proof Since M is totally geodesic, then
$$ h(X,Y) = 0 $$ for any $X, Y \in \Gamma (TM).$ \\
By definition of lightlike submanifolds, $h(X,Y) = 0$ if and only if
$$ {g}(h(X,Y), \xi) = 0 $$ and
$$ {g}(h(X,Y), W) = 0. $$

\begin{eqnarray}
 \nonumber {g}(h(X,Y), \xi) &=& {g} (\overline{\nabla}_{X} Y - \nabla_{X}Y, \xi) \\
   \nonumber  {} &=& {g} (\overline{\nabla}_{X} Y, \xi) - {g}(\nabla_{X}Y, \xi) \\
  \nonumber  {} &=& {g} ({\nabla}_{X} Y, \xi) \\
 \nonumber  {} &=& \overline{\nabla}_{X} {g} ( Y, \xi) - {g}(Y, \overline{\nabla}_{X}\xi) \\
 \nonumber  {} &=&  - {g}(Y, \overline{\nabla}_{X}\xi)\\
  \nonumber  {} &=&  - {g} (Y, [X, \xi] + \overline{\nabla}_{\xi}X \\
 \nonumber  {}&=& - {g} (Y, [X, \xi])  -  g(Y, \overline{\nabla}_{\xi}X) \\
  \nonumber  {} &=& - {g} (Y, [X, \xi])  +  {g}(\overline{\nabla}_{\xi}Y, X) \\
  \nonumber  {} &=& - {g} (Y, [X, \xi])  + {g}([\xi, Y], X) + {g}(\overline{\nabla}_{Y}\xi, X) \\
  \nonumber  {} &=&  - {g} (Y, [X, \xi])  + {g}(X, [\xi, Y]) + {g}(\overline{\nabla}_{Y}\xi, X) \\
  \nonumber  {} &=& - (L_{\xi}{g})(X,Y) + {g}(\overline{\nabla}_{Y}\xi, X) \\
  \nonumber  {} &=& - (L_{\xi}{g})(X,Y) - {g}(\xi,\overline{\nabla}_{Y} X) \\
  \nonumber  {} &=& - (L_{\xi}{g})(X,Y) - {g}(\xi,{\nabla}_{Y} X + h(X,Y)) \\
  \nonumber  {} &=& - (L_{\xi}{g})(X,Y) - {g}(\xi,{\nabla}_{Y} X) - {g}(\xi,h(X,Y))) \\
  \nonumber  {} &=& - (L_{\xi}{g})(X,Y)  - {g}(\xi, h(X,Y))) \\
  \nonumber  2{g}(h(X,Y) &=&  - (L_{\xi}{g})(X,Y).
\end{eqnarray}
Since ${g}(h(X,Y), \xi) = 0,$ we have
\begin{equation}\label{2e1}
\nonumber (L_{\xi}{g})(X,Y) = 0.
\end{equation}
Similarly,
\begin{eqnarray} \
 \nonumber {g}(h(X,Y),W) &=& {g}(\overline{\nabla}_{X}Y - \nabla_{X}Y, W) \\
 \nonumber  {} &=& {g}(\overline{\nabla}_{X}Y, W) \\
 \nonumber  {} &=& - {g}(Y, [X, W] + \overline{\nabla}_{W}X) \\
\nonumber   {} &=& - {g}(Y, [X, W]) - {g}(Y, \overline{\nabla}_{W}X) \\
 \nonumber  {} &=& - {g}(Y, [X, W]) + {g}(\overline{\nabla}_{W}Y, X) \\
 \nonumber  {} &=&  - {g} (Y, [X, W])  + {g}([W, Y], X) + {g}(\overline{\nabla}_{Y}W, X) \\
  \nonumber {} &=&  - {g} (Y, [X, W])  + {g}(X, [W, Y]) + {g}(\overline{\nabla}_{Y}W, X) \\
 \nonumber  {} &=& - (L_{W}{g})(X,Y) + {g}(\overline{\nabla}_{Y}W, X) \\
  \nonumber {} &=& - (L_{W}{g})(X,Y) - {g}(W,\overline{\nabla}_{Y} X) \\
  \nonumber {} &=& - (L_{W}{g})(X,Y) - {g}(W,{\nabla}_{Y} X + h(X,Y)) \\
 \nonumber  {} &=& - (L_{W}{g})(X,Y) - {g}(W,{\nabla}_{Y} X) - {g}(W,h(X,Y))) \\
 \nonumber  {} &=& - (L_{W}{g})(X,Y)  - {g}(W, h(X,Y)) \\
 \nonumber  2{g}(h(X,Y),W) &=& - (L_{W}{g})(X,Y).
\end{eqnarray}
Since ${g}(h(X,Y),W) = 0,$ we get

\begin{equation}\label{2e2}
\nonumber (L_{W}{g})(X,Y) = 0
\end{equation}
for any $W \in \Gamma (S(TM^{\bot})).$

\begin{lem}\label{1l1}
Let $M$ be a CR-lightlike submanifold of a golden semi-Riemannian manifold $\overline{M}.$ Then
$$ {g}(h(X,Y), W) = {g}(A_{W}X,Y) $$
for any $X \in \Gamma (D), \; Y \in \Gamma (D') \; and \; W \in \Gamma (S(TM^{\bot})).$
\end{lem}

\proof 
Using (\ref{1e5a}), we have
\begin{eqnarray}
  \nonumber {g} (h(X,Y), W) &=& {g}(\overline{\nabla}_{X}Y - \nabla_{X}Y, W) \\
  \nonumber {} &=& {g}(\overline{\nabla}_{X}Y, W) - {g}(\nabla_{X}Y, W)\\
  \nonumber {} &=& {g}(\overline{\nabla}_{X}Y, W) \\
 \nonumber {} &=& {g}(Y,\overline{\nabla}_{X}W).
\end{eqnarray}
From (\ref{1e8}) it follows that
\begin{eqnarray}
  \nonumber {g}(h(X,Y), W) &=& - {g}(Y,- A_{W}X + \nabla_{X}^{s}W + D^{s}(X, W)) \\
  \nonumber {} &=& {g}(Y, A_{W}X) - {g} (Y, \nabla_{X}^{s}W)  - {g}(Y, D^{s}(X, W)) \\
  \nonumber {g}(h(X,Y), W) &=& {g}(Y, A_{W}X),
\end{eqnarray}
where $X \in \Gamma (D), \; Y \in \Gamma (D') \; and \; W \in \Gamma (S(TM^{\bot})).$
\begin{thm}\label{2th2}
Let $M$ be a CR-lightlike submanifold of a golden semi-Riemannian manifold $\overline{M}.$ Then $M$ is mixed geodesic if and only if
$$ A^{*}_{\xi}X \in \Gamma (D_{0} \bot {J}L_{1}) $$ and
$$ A_{W}X \in \Gamma (D_{0} \bot Rad(TM) \bot {J}L_{1}) $$
for any $X \in \Gamma (D), \; \xi \in \Gamma (Rad(TM)) \; and \; W \in \Gamma (S(TM^{\bot})).$
\end{thm}
\proof 
For any $X \in \Gamma (D), \; Y \in \Gamma (D') \; and \; W \in \Gamma (S(TM^{\bot})),$ we have \\
From (\ref{1e5a}),
\begin{eqnarray}
 \nonumber {g}(h(X,Y), \xi) &=& {g}(\overline{\nabla}_{X}Y - \nabla_{X}Y, \xi) \\
  \nonumber {} &=& {g}(\overline{\nabla}_{X}Y, \xi) - {g}(\nabla_{X}Y, \xi) \\
  \nonumber {} &=& {g}(\overline{\nabla}_{X}Y, \xi ) \\
  \nonumber {} &=& - {g}(Y, \overline{\nabla}_{X}\xi).
\end{eqnarray}
Again using (\ref{1e5a}), we get
\begin{eqnarray}
  \nonumber {} {g}(h(X,Y), \xi) &=& - {g}(Y, \nabla_{X}\xi + h(X, \xi)) \\
  \nonumber {} &=& - {g}(Y, \nabla_{X}\xi) - {g}(Y, h(X, \xi)) \\
  \nonumber {} &=& - {g}(Y, \nabla_{X}\xi).
\end{eqnarray}
Using (\ref{1e10}), we have
\begin{eqnarray}
  \nonumber {} {g}(h(X,Y), \xi) &=& - {g}( Y, - A_{\xi}^{*}X + \nabla^{*t}_{X}\xi) \\
  \nonumber {} &=& {g}( Y,  A_{\xi}^{*}X) + {g}(Y, \nabla^{*t}_{X}\xi) \\
  \nonumber {g}(h(X,Y), \xi) &=& {g}( Y,  A_{\xi}^{*}X)  \\
 \nonumber {g}( Y,  A_{\xi}^{*}X)&=& 0.
\end{eqnarray}
Since CR-lightlike submanifold $M$ is mixed geodesic. Then 
$$ {g}(h(X,Y), \xi) = 0 $$
$$ \Rightarrow {g}( Y,  A_{\xi}^{*}X) = 0 $$ 
$$\Rightarrow A_{\xi}^{*}X \in \Gamma (D_{0} \bot {J}L_{1}), $$
where $X \in \Gamma (D),  Y \in \Gamma (D').$\\
From (\ref{1e5a}), we get 

\begin{eqnarray}
  \nonumber {g}(h(X,Y), W) &=& {g} (\overline{\nabla}_{X}Y - \nabla_{X}Y, W) \\
  \nonumber {} &=& {g} (\overline{\nabla}_{X}Y, W) - {g}(\nabla_{X}Y, W) \\
  \nonumber {} &=& {g} (\overline{\nabla}_{X}Y, W) \\
  \nonumber {} &=& - {g} (Y, \overline{\nabla}_{X}W).
\end{eqnarray}
From (\ref{1e8}), we get
\begin{eqnarray}
  \nonumber {} {g}(h(X,Y), W) &=& - {g} (Y, - A_{W}X + \nabla_{X}^{s}W + D^{l}(X,W))\\
  \nonumber {} &=& {g} (Y,  A_{W}X) - {g}(Y, \nabla_{X}^{s}W) - {g}(Y, D^{l}(X,W)) \\
  \nonumber {g}(h(X,Y), W) &=& {g} (Y,  A_{W}X).
\end{eqnarray}
Since, $M$ is mixed geodesic, then ${g}(h(X,Y), W) = 0$\\
$$\Rightarrow {g} (Y,  A_{W}X) = 0. $$
$$ A_{W}X \in \Gamma (D_{0} \bot Rad(TM) \bot {J}L_{1}). $$

\begin{thm}\label{2th3}
Let $M$ be a CR-lightlike submanifold of a semi-Riemannian manifold $\overline{M}.$ Then $M$ is $D'-$geodesic if and only if $A_{W}Z$ and $A_{\xi}^{*}Z$ have no component in $M_{2}\bot {J}Rad(TM)$ for any $Z \in \Gamma (D'), \; \xi \in \Gamma (Rad(TM)) \; and \; W \in \Gamma (S(TM^{\bot})).$
\end{thm}
\proof From (\ref{1e5a}) we have
\begin{eqnarray}
  \nonumber {g}(h(Z,V), W) &=& {g}(\overline{\nabla}_{Z}V - \nabla_{X}Y, W) \\
  \nonumber {}  &=& {g}(\overline{\nabla}_{Z}V, W) - {g}(\nabla_{X}Y, W) \\
  \nonumber {}  &=& - \overline{g}(\nabla_{X}Y, W),
\end{eqnarray}
where $Z, V \in \Gamma (D').$ \\
Using (\ref{1e8}), we have
\begin{eqnarray}
  \nonumber {g}(h(Z,V), W) &=& - {g}(V, - A_{W}Z + \nabla_{Z}^{s}+ D^{l}(Z,W)) \\
  \nonumber {} &=&  {g}(V,  A_{W}Z) - {g}(V, \nabla_{Z}^{s})- {g}(V, D^{l}(Z,W))\\
  {g}(h(Z,V), W) &=&  {g}(V,  A_{W}Z) \label{2e3}.
\end{eqnarray}
Since $M$ is $D'-$geodesic, then ${g}(h(Z,V),W)=0.$ \\
From (\ref{2e3}), we get\\ 
$$ {g}(V,  A_{W}Z) = 0. $$
Now,
\begin{eqnarray}
  \nonumber {g}(h(Z,V), \xi) &=& {g}(\overline{\nabla}_{Z}V - \nabla_{Z}V, \xi) \\
 \nonumber {} &=& {g}(\overline{\nabla}_{Z}V, \xi) - {g}(\nabla_{Z}V, \xi) \\
   \nonumber {} &=& {g}(\overline{\nabla}_{Z}V, \xi) =  - {g}(V, \overline{\nabla}_{Z} \xi).
\end{eqnarray}
From (\ref{1e10}), we have
\begin{eqnarray}
  \nonumber {} {g}(h(Z,V), \xi) &=& - {g}(V, - A_{\xi}^{*}Z + \nabla_{Z}^{*t}\xi) \\
 \nonumber {} &=& {g}(V,  A_{\xi}^{*}Z) - {g}(V, \nabla_{Z}^{*t}\xi) \\
  \nonumber {g}(h(Z,V), \xi) &=& {g}(V,  A_{\xi}^{*}Z) \\
\nonumber  {g}(h(Z,V), \xi) &=&  {g}( A_{\xi}^{*}V,  Z). \label{2e4}
\end{eqnarray}
Since $M$ is $D'-$ geodesic, then
$$ {g}(h(Z,V), \xi) = 0 $$ 
$$ \Rightarrow {g}( A_{\xi}^{*}V,  Z) = 0. $$
Thus, $A_{W}Z$ and $A^{*}_{\xi}Z$ have no component in $M_{2}\bot JRad(TM).$

\begin{lem}\label{2l2}
Let $M$ be a CR-lightlike submanifold of a golden semi-Riemannian manifold $\overline{M}.$ If the distribution $D$ is integrable, then the following conditions hold:\\
(i) $ J{g}(D^{l}({J}X,W),Y) - {g}(D^{l}(X,W), {J}Y)  =  {g}(A_{W}{J}X, Y) - {g}( A_{W}X,{J}Y), $ \\
(ii) $  {g}(D^{l}({J}X), \xi)  = {g}( A_{W}X, {J} \xi), $\\
(iii) ${g}(D^{l}(X,W),\xi) =  {g}(A_{W}{J}X, J\xi) - {g}( A_{W}X, J\xi),$\\
where $X, Y \in \Gamma (TM), \; \xi \in \Gamma (Rad(TM)) $ and $ W \in \Gamma (S(TM^{\bot})).$
\end{lem}
\proof From equation (\ref{1e8}), we have
\begin{eqnarray}
  \nonumber {g}(D^{l}({J}X,W),Y) &=& {g}(\overline{\nabla}_{JX}W + A_{W}{J}X - \nabla_{JX}^{s}W, Y) \\
  \nonumber {} &=& {g}(\overline{\nabla}_{JX}W,Y) + {g}(A_{W}{J}X, Y) - {g}(\nabla_{JX}^{s}W, Y) \\
  \nonumber {} &=& {g}(\overline{\nabla}_{JX}W,Y) + {g}(A_{W}{J}X, Y) \\
  \nonumber {} &=& - {g}(W,\overline{\nabla}_{JX}Y) + {g}(A_{W}{J}X, Y). 
\end{eqnarray}
 Using (\ref{1e5a}), we have
\begin{eqnarray}
  \nonumber {} {g}(D^{l}({J}X,W),Y) &=& - {g}(W,{\nabla}_{JX}Y + h({J}X,Y)) + {g}(A_{W}{J}X, Y)\\
  \nonumber {} &=& - {g}(W,{\nabla}_{JX}Y) - {g}(W, h({J}X,Y)) + {g}(A_{W}{J}X, Y) \\
  \nonumber {} &=& - {g}(W, h({J}X,Y)) + {g}(A_{W}{J}X, Y)\\
  \nonumber {} &=& - {g}(W, h(X,{J}Y)) + {g}(A_{W}{J}X, Y).
\end{eqnarray}
 Again using (\ref{1e5a}), we have
\begin{eqnarray} 
  \nonumber {} {g}(D^{l}({J}X,W),Y) &=& - {g}(W, \overline{\nabla}_{X}{J}Y - \nabla_{X}{J}Y) + {g}(A_{W}{J}X, Y) \\
  \nonumber {} &=& - {g}(W, \overline{\nabla}_{X}{J}Y) + {g}(W, \nabla_{X}{J}Y) + {g}(A_{W}{J}X, Y) \\
  \nonumber {} &=& - {g}(W, \overline{\nabla}_{X}{J}Y) + {g}(A_{W}{J}X, Y)  \\
  \nonumber {} &=& {g}(\overline{\nabla}_{X}W, {J}Y) + {g}(A_{W}{J}X, Y).
\end{eqnarray}
Using (\ref{1e8}), we get
\begin{eqnarray}
  \nonumber {} {g}(D^{l}({J}X,W),Y) &=& {g}(- A_{W}X + \nabla_{X}^{s}W + D^{l}(X,W), {J}Y) +  \\
  \nonumber {} &=& {g}(A_{W}{J}X, Y) - {g}( A_{W}X,{J}Y)  + {g}(\nabla_{X}^{s}W, {J}Y)  + \\ &&
   \nonumber {g}(D^{l}(X,W), {J}Y) + {g}(A_{W}{J}X, Y) \\
  \nonumber {} &=& - {g}( A_{W}X,{J}Y) +  {g}(A_{W}{J}X, Y)  \\
  \nonumber {} {g}(D^{l}({J}X,W),Y) - {g}(D^{l}(X,W), {J}Y)  &=& {g}(A_{W}{J}X, Y) - {g}( A_{W}X,{J}Y).
\end{eqnarray}

(ii) Using (\ref{1e8}), we have
\begin{eqnarray}
  \nonumber {g}(D^{l}({J}X, W), \xi) &=& {g}( A_{W}{J}X - \nabla_{{J}X}^{s}W + {\nabla}_{{J}X}W, \xi ) \\
  \nonumber {} &=& {g}( A_{W}{J}X, \xi) - {g}(\nabla_{{J}X}^{s}W, \xi) + \overline{g}({\nabla}_{{J}X}W, \xi ) \\
  \nonumber {} &=& {g}( A_{W}{J}X, \xi) + {g}(\overline{\nabla}_{{J}X}W, \xi ) \\
  \nonumber {} &=& {g}( A_{W}{J}X, \xi) - {g}(W, \overline{\nabla}_{{J}X}\xi ).
\end{eqnarray}
 Using (\ref{1e10}), we get
\begin{eqnarray}
 \nonumber {} {g}(D^{l}({J}X, W), \xi) &=& {g}( A_{W}{J}X, \xi) + {g}(W,  A^{*}_{\xi}{J}X ) - {g}(W, \nabla^{*t}_{{J}X}, \xi) \\
  \nonumber {g}(D^{l}({J}X), \xi) &=& {g}( A_{W}{J}X, \xi) \\
  \nonumber {g}(D^{l}({J}X), \xi)  &=& {g}( A_{W}X, {J} \xi).
\end{eqnarray}
(iii) Replacing $Y$ by $J\xi$ in (i), we have
\begin{eqnarray}
\nonumber J{g}(D^{l}({J}X,W),J\xi) - {g}(D^{l}(X,W), {J^{2}}\xi)  =  {g}(A_{W}{J}X, J\xi) - {g}( A_{W}X,{J^{2}}\xi). 
\end{eqnarray}
Using (\ref{1e16a}), we get
\begin{eqnarray}
	\nonumber J{g}(D^{l}({J}X,W),J\xi) - {g}(D^{l}(X,W), (J+I)\xi) & = & {g}(A_{W}{J}X, J\xi) - k,{g}( A_{W}X,(J+I)\xi)\\
\nonumber J{g}(D^{l}({J}X,W),J\xi) - {g}(D^{l}(X,W), J\xi) - {g}(D^{l}(X,W),\xi) & = & {g}(A_{W}{J}X, J\xi) - {g}( A_{W}X, J\xi) - \\ && \nonumber  {g}( A_{W}X, \xi).\\
\nonumber {g}(D^{l}(X,W),\xi) & = & {g}(A_{W}{J}X, J\xi) - {g}( A_{W}X, J\xi).
\end{eqnarray}

\section{Totally Umbilical CR-lightlike submanifolds.}
\begin{defn}
A lightlike submanifold $(M, g) $ of a semi-Riemannian manifold $(\overline{M}, {g})$ is said to be totally umbilical in $\overline{M}$ if there is a smooth transversal vector field $H \in \Gamma (tr(TM))$ on $M,$ called the transversal curvature vector field of $M,$ such that 
$$ h(X,Y) = H{g}(X,Y), $$ 
where $X, Y \in \Gamma (TM).$\\
Using (\ref{1e6}), $M$ is a totally umbilical if and only if there exists smooth vector fields $H^{l} \in \Gamma (ltr(TM))$ and $H^{s}(X, Y) \in \Gamma (S(TM^{\bot})),$ such that
$$ h^{l}(X,Y) = H^{l}{g}(X,Y), \; \; h^{s}(X,Y) = H^{s}{g}(X,Y).$$
\end{defn}

\begin{thm}
Let $M$ be a totally umbilical CR-lightlike submanifold of a locally golden manifold $\overline{M}$ and screen distribution be totally geodesic. Then
$$A_{JZ}W  = A_{JW}Z, \; \forall \; W,Z \in \Gamma D'.$$
\end{thm}
\begin{proof}
Since $\overline{M}$ be a golden semi-Riemannian manifold, therefore 
$$J\overline{\nabla}_{Z}W = \overline{\nabla}_{Z}JW. $$	
Using (\ref{1e5a}) and (\ref{1e5b}), we have
\begin{eqnarray}
\nonumber J(\nabla_{Z}W + h(Z,W)) & = &- A_{JW}Z + \nabla_{Z}^{t}JW \\
 J(\nabla_{Z}W) + J(h(Z,W)) & = & - A_{JW}Z + \nabla_{Z}^{t}JW.\label{4e1}
\end{eqnarray} 
Interchange $Z$ and $W$ in the above equation, we have 
\begin{equation} \label{4e2}
	J(\nabla_{W}Z) + J(h(W,Z)) = - A_{JZ}W + \nabla_{W}^{t}JZ.	
\end{equation}
Substract equations (\ref{4e1}) and (\ref{4e2}), we get
\begin{eqnarray}
\nonumber J(\nabla_{Z}W) - J(\nabla_{W}Z) + J(h(Z,W)) - J(h(W,Z)) & = & A_{JZ}W - A_{JW}Z + \nabla_{Z}^{t}JW  \\ && \nonumber - \nabla_{W}^{t}JZ\\
\nonumber J(\nabla_{Z}W) - J(\nabla_{W}Z) & = & A_{JZ}W - A_{JW}Z + \nabla_{Z}^{t}JW \\ && \nonumber - \nabla_{W}^{t}JZ \\
J(\nabla_{Z}W - \nabla_{W}Z) - \nabla_{Z}^{t}JW + \nabla_{W}^{t}JZ & = & A_{JZ}W - A_{JW}Z. \label{4e3}
\end{eqnarray}
Taking inner product with $X \in \Gamma (D_{0})$ in (\ref{4e3}), we have 
\begin{eqnarray} 
\nonumber {g}(J\nabla_{Z}W, X) - {g}(J(\nabla_{W}Z,X) - {g}(\nabla_{W}^{t}JZ,X)-{g}(\nabla_{W}^{t}JZ,X) & = &{g}(A_{JZ}W,X) \\ && \nonumber -{g}(A_{JW}Z,X)\\
\nonumber {g}(J(\nabla_{W}Z,X) - {g}(J(\nabla_{W}Z,X) & = & {g}(A_{JZ}W,X) \\ && \nonumber -{g}(A_{JW}Z,X).
\end{eqnarray}
Using (\ref{1e16b}), we get
\begin{eqnarray}
\nonumber {g}(\nabla_{W}Z, JX) - {g}(\nabla_{W}Z, JX) & = & {g}(A_{JZ}W,X)-{g}(A_{JW}Z,X)\\
{g}(A_{JZ}W - A_{JW}Z, X) & = & {g}(\nabla_{W}Z, JX) - {g}(\nabla_{W}Z, JX).\label{4e4}
\end{eqnarray}
Now, 
\begin{eqnarray}
\nonumber {g}(\nabla_{W}Z, JX) & = & {g}(\overline{\nabla}_{W}Z - h(W,Z), JX) \\
\nonumber {} & = & {g}(\overline{\nabla}_{W}Z, JX) - {g}(h(W,Z), JX),\\
\nonumber {} & = & {g}(\overline{\nabla}_{W}Z, JX)\\
\nonumber {} & = & - {g}(Z, {\nabla}_{W}JX)\\
\nonumber {g}(\nabla_{W}Z, JX) & = & - {g}(Z, (\overline{\nabla}_{W}J)X - J(\overline{\nabla}_{W}X)).
\end{eqnarray}
Since $J$ is parallel to $\overline{\nabla},$ i.e. $\overline{\nabla}_{X}J = 0$ 
$$ {g}(\nabla_{W}Z, JX)  =  - J(\overline{\nabla}_{W}X)).  $$ 
Using (\ref{1e6}), we have
\begin{eqnarray}
\nonumber {g}(\nabla_{W}Z, JX) & =  &- {g} (JZ, \nabla_{W}X + h^{s}(W,X) + h^{l}(W,X))\\
 {g}(\nabla_{W}Z, JX) & = & - {g} (JZ, \nabla_{W}X) - {g}(JZ, h^{s}(W,X))  - {g}(JZ, h^{l}(W,X)). \label{4e5}
\end{eqnarray}
Since $M$ is a totally umbilical CR-lightlike submanifold and screen distribution is totally geodesic, then 
$$h^{s}(W,X) = H^{s}{g}(W,X) = 0 $$ and 
$$h^{l}(W,X) = H^{l}{g}(W,X) = 0, $$
where $W \in \Gamma (D')$ and $X \in \Gamma (D_{0}).$ \\
From (\ref{4e5}), we have 
$$ {g}(\nabla_{W}Z, JX)  =  - {g} (JZ, \nabla_{W}X). $$
From Lemma \ref{2l4}, we get
$$ {g}(\nabla_{W}Z, JX)  = 0. $$
Similarly, $$ {g}(\nabla_{Z}W, JX)  = 0 $$
Using (\ref{4e4}), we have 
$$  g(A_{JZ}W - A_{JW}Z, X) = 0. $$
Since $D_{0}$ is non-degenerate, then 
$$ A_{JZ}W - A_{JW}Z = 0 $$
$$ \Rightarrow A_{JZ}W = A_{JW}Z. $$
\end{proof}
\begin{defn}
For a CR-lightlike submanifold $M,$ a plane $X\wedge Z$ with $X \in \Gamma (D_{0})$ and $Z \in \Gamma (D')$ is called a CR-lightlike section. The sectional curvature $\overline{K} (\pi)$ of a CR-lightlike section $\pi$ is called CR-lightlike sectional curvature.
\end{defn}
\begin{thm}
Let $M$ be a totally umbilical CR-lightlike submanifold $ M$ of a golden manifold $\overline{M}.$ Then, the CR-lightlike sectional curvature of $M$ vanishes, that is $\overline{K}(\pi)=0$ for all CR-lightlike section $\pi.$
\end{thm}

\begin{proof}
Since, $M$ is a totally umbilical CR-lightlike submanifold of $\overline{M}$ then from (\ref{1e10f}) and (\ref{1e10g}), we have 
\begin{eqnarray}
(\nabla_{X}h^{l})(Y,Z) = g(Y,Z)\nabla^{l}_{X}H^{l}-H^{l}\{ (\nabla_{X}g)(Y,Z)\}\label{4e6},\\
(\nabla_{X}h^{s})(Y,Z) = g(Y,Z)\nabla^{s}_{X}H^{s}-H^{s}\{ (\nabla_{X}g)(Y,Z)\}\label{4e7}.
\end{eqnarray}
For a CR-lightlike section $\pi = X\wedge Z, \; X \in \Gamma (D_{0}), Z \in \Gamma (D'). $\\
From (\ref{1e10e}), we have $(\nabla_{X}g)(Y,Z)= 0.$ Therefore, from (\ref{4e6})  and (\ref{4e7}), we get
\begin{eqnarray}
(\nabla_{X}h^{l})(Y,Z) = g(Y,Z)\nabla^{l}_{X}H^{l}\label{4e8},\\
(\nabla_{X}h^{s})(Y,Z) = g(Y,Z)\nabla^{s}_{X}H^{s}\label{4e9}.
\end{eqnarray}
Now, from (\ref{4e8}) and (\ref{4e9}), we get 
\begin{eqnarray}
\nonumber \{\overline{R}(X,Y)Z \}^{tr} & = & g(Y,Z)\nabla^{l}_{X} H^{l}- g(X,Z)\nabla^{l}_{Y} H^{l} + g(Y,Z)D^{l}(X,H^{s}) \\ && 
\nonumber - g(X,Z)D^{l}(Y,H^{s})+g(Y,Z)\nabla^{s}_{X} H^{s}-g(X,Z)\nabla^{s}_{Y}H^{s} \\ && 
 +g(Y,Z)D^{s}(X,H^{l})-g(X,Z)D^{s}(Y,H^{l}).\label{4e10}
\end{eqnarray}
For any $V \in \Gamma (tr(TM))$, from equation (\ref{4e10}) we get
\begin{eqnarray}
\nonumber \overline{R}(X,Y,Z,V ) & = & g(Y,Z){g}(\nabla^{l}_{X} H^{l}, V)- g(X,Z){g}(\nabla^{l}_{Y} H^{l}, V) + \\ && 
\nonumber   g(Y,Z){g}(D^{l}(X,H^{s}), V) - g(X,Z){g}(D^{l}(Y,H^{s}), V)+  \\ && 
\nonumber  g(Y,Z) {g}(\nabla^{s}_{X} H^{s}, V)-g(X,Z){g}(\nabla^{s}_{Y}H^{s}, V) \\ && \nonumber  +g(Y,Z){g}(D^{s}(X,H^{l}), V)- g(X,Z){g}(D^{s}(Y,H^{l}, V).\label{4e11} 
\end{eqnarray}
\begin{eqnarray}
	\nonumber \overline{R}(X,Z,X,Z ) & = & g(Z,X){g}(\nabla^{l}_{X} H^{l}, Z)- g(X,X){g}(\nabla^{l}_{Z} H^{l}, Z)  \\ && 
	\nonumber + g(Z,X){g}(D^{l}(X,H^{s}), Z) - g(X,X){g}(D^{l}(Z,H^{s}), Z)+  \\ && 
\nonumber g(Z,X) {g}(\nabla^{s}_{X} H^{s}, Z)-g(X,X){g}(\nabla^{s}_{Z}H^{s}, Z)	 \\ && 
\nonumber  +g(Z,X){g}(D^{s}(X,H^{l}), Z)- g(X,X){g}(D^{s}(Z,H^{l}, Z).\label{4e12}
\end{eqnarray}
\begin{eqnarray}
\nonumber {R}(X,Z,JX, JZ ) & = & g(Z, JX){g}(\nabla^{l}_{X} H^{l}, JZ)- g(X,JX){g}(\nabla^{l}_{Z} H^{l}, JZ) + \\ && 
\nonumber g(Z,JX){g}(D^{l}(X,H^{s}), JZ)  
 - g(X,JX){g}(D^{l}(Z,H^{s}), JZ)+ \\ && 
\nonumber g(Z, JX) {g}(\nabla^{s}_{X} H^{s},JZ)-g(X,JX){g}(\nabla^{s}_{Z}H^{s}, JZ) + \\ &&
\nonumber g(Z, JX){g}(D^{s}(X,H^{l}), JZ)- g(X, JX){g}(D^{s}(Z,H^{l}, JZ). \label{4e13}	
\end{eqnarray}

For any unit vectors $X \in \Gamma (D)$ and $Z \in \Gamma (D'),$ we have 
$$ \overline{R}(X, Z, JX, JZ) = \overline{R}(X, Z, X, Z) = 0. $$ 
Since, $$K(\gamma) = K_{N}(X \wedge Y) = g(\overline{R}(X,Y)Y, X),$$
where $$\overline{R}(X, Z, X, Z) = g(\overline{R}(X,Z)X, Z) $$
or,
$$ \overline{R}(X, Z, JX, JZ) = g(\overline{R}(X,Z)JX, JZ)  $$
i.e. $$\overline{K}(\pi) = 0 $$ for all CR-sections $\pi.$

\section{Example}

{\bf Example 5.1 } We consider a semi-Riemannian manifold  $R_{2}^{6}$ of signature $(-, -, +, +, +, +) $ with respect to the canonical basis $(\frac{\partial}{\partial x_{1}}, \frac{\partial}{\partial x_{2}}, \frac{\partial}{\partial x_{3}}, \frac{\partial}{\partial x_{4}}, \frac{\partial}{\partial x_{5}} , \frac{\partial}{\partial x_{6}} ) $ and a submanifold $M$ of codimension $2$ in $R_{2}^{6}$ given by equations 
$$ x_{5} = x_{1}cos\alpha - x_{2}sin\alpha - x_{3}x_{4} tan\alpha $$
$$ x_{6} = x_{1}sin\alpha - x_{2}cos\alpha - x_{3}x_{4}, $$
where $\alpha \in R - \{ \frac{\pi}{2} + k\pi; \ k \in z \}.$ The structure on $R^{6}_{2}$ is defined by 
$$ J ( \frac{\partial}{\partial x_{1}}, \frac{\partial}{\partial x_{2}}, \frac{\partial}{\partial x_{3}}, \frac{\partial}{\partial x_{4}}, \frac{\partial}{\partial x_{5}} , \frac{\partial}{\partial x_{6}})  = (\overline{\phi} \ \frac{\partial}{\partial x_{1}}, \overline{\phi} \frac{\partial}{\partial x_{2}}, \phi \frac{\partial}{\partial x_{3}}, \phi \frac{\partial}{\partial x_{4}}, \phi \frac{\partial}{\partial x_{5}}, \phi \frac{\partial}{\partial x_{6}} ).$$ 

Now,
$$ J^{2} ( \frac{\partial}{\partial x_{1}}, \frac{\partial}{\partial x_{2}}, \frac{\partial}{\partial x_{3}}, \frac{\partial}{\partial x_{4}}, \frac{\partial}{\partial x_{5}} , \frac{\partial}{\partial x_{6}})  = ((\overline{\phi}+ 1) \ \frac{\partial}{\partial x_{1}}, (\overline{\phi} + 1) \frac{\partial}{\partial x_{2}}, (\phi + 1) \frac{\partial}{\partial x_{3}}, (\phi + 1) \frac{\partial}{\partial x_{4}}, $$ $$ ( \phi + 1) \frac{\partial}{\partial x_{5}}, ( \phi + 1) \frac{\partial}{\partial x_{6}} )   $$ 
$$ J^{2} ( \frac{\partial}{\partial x_{1}}, \frac{\partial}{\partial x_{2}}, \frac{\partial}{\partial x_{3}}, \frac{\partial}{\partial x_{4}}, \frac{\partial}{\partial x_{5}} , \frac{\partial}{\partial x_{6}})  = J ( \frac{\partial}{\partial x_{1}}, \frac{\partial}{\partial x_{2}}, \frac{\partial}{\partial x_{3}}, \frac{\partial}{\partial x_{4}}, \frac{\partial}{\partial x_{5}} , \frac{\partial}{\partial x_{6}}) \  + \  ( \frac{\partial}{\partial x_{1}}, \frac{\partial}{\partial x_{2}}, $$ $$ \frac{\partial}{\partial x_{3}}, \frac{\partial}{\partial x_{4}}, \frac{\partial}{\partial x_{5}} , \frac{\partial}{\partial x_{6}}).$$ 
  $$ J^{2} = J + I. $$ 
  It follows that $( R_{2}^{6}, J  )$ is a golden semi-Reimannian manifold. \\

The tangent bundle $TM$ is spanned by
$$ X_{0} = - sin\alpha \ \frac{\partial}{\partial x_{5}} -  cos\alpha \ \frac{\partial}{\partial x_{6}} - \phi \ \frac{\partial}{\partial x_{2}}, $$  
$$ X_{1} = - \phi \ sin\alpha \ \frac{\partial}{\partial x_{5}} -  \phi \ cos\alpha \ \frac{\partial}{\partial x_{6}} +   \ \frac{\partial}{\partial x_{2}}, $$  
$$ X_{2} = \frac{\partial}{\partial x_{5}} - \overline{\phi} \ sin\alpha \ \frac{\partial}{\partial x_{2}} + \frac{\partial}{\partial x_{1}}, $$ 
$$ X_{3} = - \overline{\phi} \ cos\alpha \ \frac{\partial}{\partial x_{2}} + \frac{\partial}{\partial x_{4}} + i \frac{\partial}{\partial x_{6}}. $$

Thus, $M$ is a 1-lightlike submanifold of $R_{2}^{6}$ with $RadTM = Span\{X_{0}\}$. Using golden structure of $R_{2}^{6},$ we obtain that $X_{1} = J(X_{0}).$ Thus $J(RadTM)$ is a distribution on $M.$ Hence $M$ is a 
CR-lightlike submanifold. 
\vspace{0.45cm}

\noindent
{\bf Acknowledgement.} All authors would like to thank Integral University, Lucknow, India for providing financial support under the seed money project IUL/IIRC/SMP/2021/010  to the present work (manuscript number IU/R$\&$D/2022-MCN0001708).\\

\vspace{.45cm}
\parindent=0mm

 \end{proof}

\end{document}